\documentclass[11pt,reqno]{article}

\usepackage[usenames]{color}
\usepackage{graphicx}
\usepackage{amscd}
\usepackage[english]{babel}
\usepackage{color}
\usepackage{fullpage}
\usepackage{psfig}
\usepackage{graphics,amsmath,amssymb}
\usepackage{amsthm}
\usepackage{amsfonts}
\usepackage{latexsym}
\usepackage{epsf}
\usepackage{float}

\usepackage[colorlinks=true,
linkcolor=webgreen,
filecolor=webbrown,
citecolor=webgreen]{hyperref}

\definecolor{webgreen}{rgb}{0,.5,0}
\definecolor{webbrown}{rgb}{.6,0,0}

\newcommand{\seqnum}[1]{\href{http://oeis.org/#1}{\underline{#1}}}
\newcommand{\beql}[1]{\begin{equation}\label{#1}}
\newcommand{\eeq}{\end{equation}}
\newcommand{\eqn}[1]{(\ref{#1})}
\newcommand{\sL}{{\cal{L}}}
\newcommand{\sR}{{\cal{R}}}

\DeclareMathOperator{\GF}{GF}
\DeclareMathOperator{\height}{ht}

\newcommand{\ZZ}{\mathbb Z}

\newtheorem{thm}{Theorem}{\bfseries}{\itshape}
{\bfseries}{\itshape}
{\bfseries}{\itshape}
{\bfseries}{\itshape}
{\bfseries}{\itshape}

\begin{document}
\theoremstyle{plain}

\begin{center}
{\large\bf Odd-Rule Cellular Automata on the Square Grid } \\
\vspace*{+.2in}

Shalosh B. Ekhad \\
C/o Doron Zeilberger \\
Department of Mathematics \\ 
Rutgers University (New Brunswick) \\
Hill Center-Busch Campus, 110 Frelinghuysen Rd. \\
Piscataway, NJ 08854-8019, USA \\
\ \\
N. J. A. Sloane\footnote{To whom correspondence should be addressed.} \\
The OEIS Foundation Inc. \\
11 South Adelaide Ave. \\
Highland Park, NJ 08904-1601, USA \\
\href{mailto:njasloane@gmail.com}{\tt njasloane@gmail.com} \\
\ \\ 
Doron Zeilberger \\
Department of Mathematics \\ 
Rutgers University (New Brunswick) \\
Hill Center-Busch Campus, 110 Frelinghuysen Rd. \\
Piscataway, NJ 08854-8019, USA \\
\href{mailto:zeilberg@math.rutgers.edu}{\tt zeilberg@math.rutgers.edu} \\
\ \\

March 13, 2015 \\
\ \\

{\bf Abstract}
\end{center}

An ``odd-rule'' cellular automaton (CA) is defined by specifying
a neighborhood for each cell, with the rule that a cell turns ON
if it is in the neighborhood of an odd number of ON cells at
the previous generation, and otherwise turns OFF.
We classify all the odd-rule CAs
defined by neighborhoods which are subsets of a $3 \times 3$
grid of square cells. There are $86$ different CAs
modulo trivial symmetries.
When we consider only the different sequences giving the number
 of ON cells after $n$ generations,
the number drops to $48$, two of which are the
Moore and von Neumann CAs.
This classification is carried out by using
the ``meta-algorithm'' described in an earlier paper to derive the 
generating functions for the $86$ sequences, and then removing duplicates.
The fastest-growing of these CAs is neither
the Fredkin nor von Neumann neighborhood,
but instead is one defined by ``Odd-rule'' 365, which turns ON almost $75\%$
of all possible cells.
 
\section{Introduction}\label{Sec1}

As in \cite{Tooth, ESZ, Fredkin}, our goal is to study how fast activity spreads in
cellular automata (CAs): more precisely, if we start with a single ON cell, how many cells 
will be ON after $n$ generations? 
For additional background see
 \cite{ Epp09, Fred2000, Kari, MOW84,
PaWo85,  Ulam62,
Wolf83, Wolf84, NKS}.
 
Continuing the investigations begun in \cite{ESZ, Fredkin}, we consider 
``odd-rule'' CAs,  concentrating on the two-dimensional rules 
defined by neighborhoods that are subsets of a $3 \times 3$ square grid.
This family of CAs includes two that were the main subject of \cite{Fredkin}, namely
Fredkin's Replicator, which is based on the Moore neighborhood,
and another which is based on the von Neumann neighborhood with a center cell.
One of the goals of the present paper is to
use the meta-algorithm from our paper \cite{ESZ}
to obtain generating functions, with proofs,
for all these sequences. This provides alternative
(computer-generated)  proofs of Theorems 4 and 5 of \cite{Fredkin}.
Another member of this family is the CA defined by the one-dimensional
Rule 150 in the Wolfram numbering scheme
\cite{PaWo85, Wolf84, NKS}.

The Wolfram numbering scheme is not, however,
particularly convenient for dealing with these $3 \times 3$ neighborhoods,
and in \S\ref{Sec3} we introduce
a simpler numbering scheme based on reading
the neighborhood as a triple of octal numbers.

Section \ref{Sec2} gives the definitions of an odd-rule
cellular automaton and the run length transform,
and quotes two essential theorems from \cite{Fredkin}.
Section \ref{Sec3} classifies odd-rule CAs that are
defined by neighborhoods that are subsets of the $3 \times 3$ grid:
if we ignore trivial differences there are $86$ different CAs
(Theorem \ref{Th3}),  shown in
Figs.\ \ref{Fig1a}, \ref{Fig1b}, \ref{Fig1c}
and Tables \ref{Tab1a}, \ref{Tab1b}, \ref{Tab1c}.
In Section \ref{Sec4} we define two CAs to be ``combinatorially
equivalent'' if the numbers of ON cells after $n$ generations
are the same for all $n$. 
Up to combinatorial equivalence there are
$48$ different CAs (Theorem \ref{Th4}),  shown in 
Tables \ref{Tab1a}, \ref{Tab1b}, \ref{Tab1c}.
In Section \ref{Sec5} we discuss three further topics: which CA has the greatest growth 
 rate (\S\ref{Sec51} -- the answer is unexpected),
which has the slowest growth rate (\S\ref{Sec52}), and
the question of explaining
why certain pairs of CAs turn out to
have the same generating function (\S\ref{Sec53}).
The 48 distinct generating functions are given
in an Appendix.

\section{Odd-rule CAs}\label{Sec2}

We consider cellular automata  whose cells are centered at the points of the
$2$-dimensional square lattice $\ZZ^2$. 
Each cell is either ON or OFF, and an ON cell
with center at the lattice point
$(i,j) \in \ZZ^2$ 
will be identified with the monomial
$x^i y^j $, which we regard as an element of the ring of Laurent
polynomials $\sR := \GF(2)[x, x^{-1}, y, y^{-1}]$
with mod 2 coefficients.  
The state of the CA is  specified by giving
the formal sum $S$ of all its ON cells.  As long
as only finitely many cells are ON, $S$ is indeed 
an element of $\sR$.

An ``odd-rule'' CA (this name was introduced in \cite{Fredkin},
although of course the concept has been
known for as long as people have been studying CAs)
 is defined by first specifying 
a neighborhood of the cell at the origin, given by
an element $F \in \sR$ listing the cells in the neighborhood.
A typical example is the Moore neighborhood,
which consists of the eight cells surrounding the cell
at the origin  (see Odd-rule $757$ in  Fig.~\ref{Fig1c}), 
and is specified by
\begin{align}\label{EqFred0}
F  &~:=~ \frac{1}{xy} + \frac{1}{y} + \frac{x}{y}
+ \frac{1}{x}+x
+\frac{y}{x}+y+xy \nonumber \\
& ~=~\left(\frac{1}{x}+1+x\right)\left(\frac{1}{y}+1+y\right) -1 ~~ \in \sR
\end{align}
The neighborhood of an arbitrary cell $x^r y^s$ is obtained
by shifting $F$ so it is centered at that cell, that is, by the product
$x^r y^s F \in \sR$. Given $F$, the corresponding
{\em odd-rule} CA is defined by the rule that a cell $x^r y^s$ is ON
at generation $n+1$ if it is the neighbor of an odd number of cells that were
ON at generation $n$, and is otherwise OFF.

Our goal is to find  $a_n(F)$, the number of ON cells at the $n$th generation
when the CA is started in generation 0 with a single ON cell at the origin.
For odd-rule CAs there is a simple formula for $a_n(F)$.
The number of nonzero terms in an element $P \in \sR$ will
be denoted by $|P|$.

\begin{thm}\label{Th1} \cite{Fredkin}
For an odd-rule CA with neighborhood $F$, 
the state at generation $n$ is equal to $F^n$, and 
$a_n(F) = |F^n|$.
\end{thm}

The sequences $[a_n(F), n \ge 0]$ are most easily described
using the ``run length transform'', an operation on number sequences 
also introduced in \cite{Fredkin}.
For an integer $n \ge 0$,
let $\sL(n)$ denote the list of the lengths of
the maximal runs of 1s in the binary expansion of $n$.
For example, since the binary expansion of 55 is 110111,
containing runs of 1s of lengths 2 and 3,
$\sL(55) = [2,3]$.
$\sL(0)$ is the empty list, and
$\sL(n)$ for $n=1,\ldots,12$ is respectively
$[1], [1], [2], [1], [1,1], [2], [3], [1], [1,1], [1,1], [1,2], [2]$
(\seqnum{A245562}\footnote{Six-digit numbers prefixed by A
refer to entries in \cite{OEIS}.}).

\vspace*{+.1in}
\noindent
\textbf{Definition.} The {\em run length transform} of a sequence 
$[S_n, n \ge 0]$ is the sequence $[T_n, n \ge 0]$ 
given by
\beql{EqRLT1}
T_n ~=~ \prod_{i \in \sL(n)} S_i.
\eeq
Note that $T_n$ depends only on the lengths of the runs of 1s 
in the binary expansion of $n$, not on the order
in which they appear. For example, since $\sL(11) = [1,2]$
and $\sL(13)=[2,1]$, $T_{11}=T_{13}= S_1S_2$.
Also $T_0=1$ (the empty product), so the value 
of $S_0$ is never used, and  will usually be taken to be 1.
For further properties and additional examples of the run
length transform see \cite{Fredkin}. 
See especially \cite[Table\ 4]{Fredkin}, which shows 
how the transformed sequence has a natural division into blocks
of successive lengths $1,1,2,4,8,16,32,\ldots$.
 
Define the {\em height} $\height(F)$ of an element
$F \in \sR$ to be $\max\{|i|, |j|\}$ for any monomial $x^i y^j$ in $F$.
If $\height(F)=h$, the cells in $F$ are a subset of the cells
in a $(2h+1) \times (2h+1)$ array of squares centered at the origin.
In particular, if $\height(F) \le 1$, we have the following:
\begin{thm}\label{Th2} \cite{Fredkin}
If the neighborhood $F$ is a subset of
the $3 \times 3$ grid of cells centered at
the origin,  then $[a_n(F), n \ge 0]$ is
the run length transform of the subsequence
$[b_n(F), n \ge 0]$, where $b_n(F) := a_{2^n-1}(F)$. 
\end{thm}

\begin{figure}[!ht]
\centerline{\includegraphics[angle=0, width=4.5in]{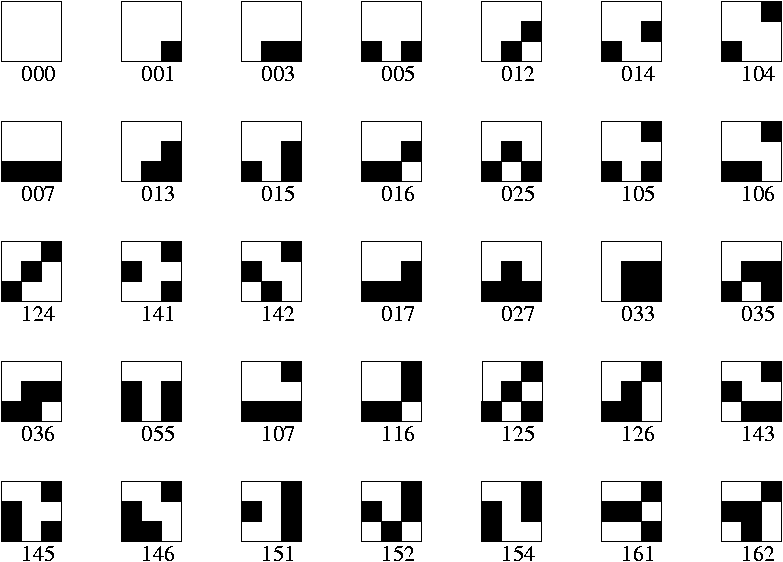}}
\caption{Up to trivial equivalence, there are $86$ distinct height-one
neighborhoods, shown in Figs.\ \ref{Fig1a}, \ref{Fig1b}, \ref{Fig1c}
together with their canonical Odd-rule numbers.}
\label{Fig1a}
\end{figure}

\begin{figure}[!ht]
\centerline{\includegraphics[angle=0, width=4.5in]{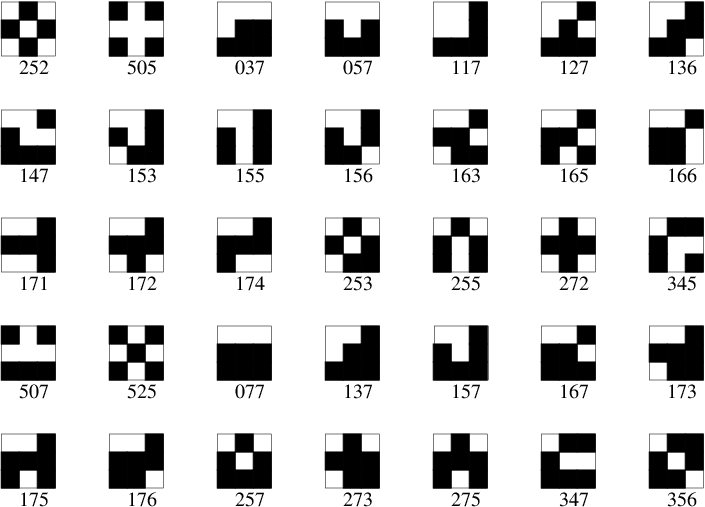}}
\caption{See caption to Fig.\ \ref{Fig1a}.}
\label{Fig1b}
\end{figure}

\begin{figure}[!ht]
\centerline{\includegraphics[angle=0, width=4.5in]{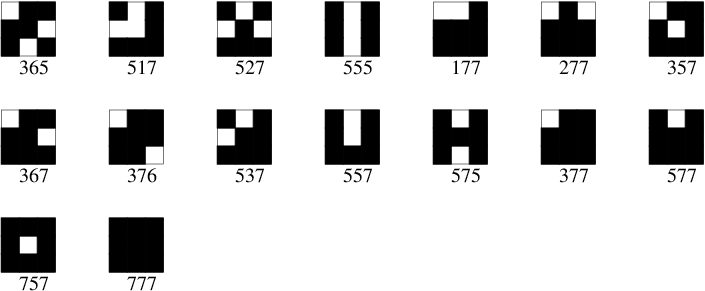}}
\caption{See caption to Fig.\ \ref{Fig1a}.}
\label{Fig1c}
\end{figure}

\section{Trivially equivalent neighborhoods}\label{Sec3}

From now on we assume that $F$ has height at most one, i.e., is a
subset of  the $3 \times 3$ grid of cells centered at
the origin. In view of Theorem \ref{Th1},
$a_n(F)$ is unchanged if we multiply (or divide) $F$ by $x$ or $y$ (these operations 
simply translate the configuration of ON states in the $(x,y)$-plane).

We can also apply any of the eight symmetries of the square 
(rotations and/or reflections, forming the dihedral group of order eight)
to $F$ without changing $a_n(F)$.

We therefore say that two neighborhoods $F \in \sR$,  $G \in \sR$ are
{\it trivially} (or {\it affinely}) equivalent
if one can be translated into the other by repeated translations, rotations, and reflections.

\begin{thm}\label{Th3} 
Up to trivial equivalence, there are $86$ 
distinct height-one neighborhoods,
as shown in Figs. \ref{Fig1a}, \ref{Fig1b}, \ref{Fig1c},
and again in Tables \ref{Tab1a}, \ref{Tab1b}, \ref{Tab1c}.
\end{thm}
\begin{proof}
Hand calculation, followed by computer verification.
\end{proof}

Rather than use the Wolfram numbering scheme, which here
could involve numbers as large as $2^{512}$, we describe
the neighborhood $F$ by a three-digit octal
number, the ``Odd-rule'' number, obtained by reading the ON cells in the $3 \times 3$
grid from left to right, top to bottom. 

The {\it canonical} Odd-rule number
for $F$ is then the smallest of the Odd-rule numbers associated with
any neighborhood that is trivially equivalent to $F$.

For example, the neighborhood $F = 1+x \in \sR$, consisting of two adjacent cells,
can be shifted or rotated into 12 different positions,
described by the octal numbers 600, 300, 060, 030, 006, 003,
440, 044, 220, 022, 110, 011.  The smallest of these is 003
(corresponding to $1/y + x/y$), which is therefore the canonical Odd-rule
number for this $F$
(see the the third figure in Fig.\ \ref{Fig1a}).

The Odd-rule number for Wolfram's one-dimensional Rule $150$ is $007$.
The two CAs that were the main subject of \cite{Fredkin}, 
namely ``Fredkin's Replicator'', which is
based on the Moore neighborhood,
and the CA based on the 
von Neumann neighborhood with a center cell,
are Odd-rules 757 and 272, respectively.
The von Neumann neighborhood without the center cell
is Odd-rule 252, and the full $3 \times 3$ neighborhood is Odd-rule $777$.

The canonical Odd-rule numbers for all 86 trivially inequivalent  height-one
neighborhoods  are shown in Figs.\ \ref{Fig1a}, \ref{Fig1b}, \ref{Fig1c},
which give graphical representations of the neighborhoods.
These 86  neighborhoods are also shown in
Tables \ref{Tab1a}, \ref{Tab1b}, \ref{Tab1c}.
The first column of these tables gives the canonical Odd-rule
number, the second column gives the number of cells in the neighborhood,
the third column gives the binary representation of the neighborhood, 
and the fourth column gives the  corresponding Laurent
polynomial $F$.

\section{ Combinatorially equivalent neighborhoods }\label{Sec4}

Since we are mostly interested in the sequences that give the
 number of ON cells after $n$ generations,
we shall say that two height-one neighborhoods $F$ and $G$
are {\it combinatorially equivalent} if $a_n(F) = a_n(G)$ for all $n \ge 0$.
In view of Theorem \ref{Th2}, an equivalent condition
is that $b_n(F) = b_n(G)$ for all $n \ge 0$.

\begin{thm}\label{Th4} 
Up to combinatorial equivalence, there are $48$ 
distinct height-one neighborhoods.
\end{thm}
\begin{proof}
Using the Maple programs \texttt{ARLT} and \texttt{GFsP}
described in \cite{ESZ} (available from \cite{CAcount}),
we computed the generating function for the $b_n(F)$ sequence
corresponding to each of the 86 neighborhoods
listed in Theorem \ref{Th3}. After removing duplicates,
48 remained (see the Appendix).
\end{proof}

As representative for each equivalence class of neighborhoods we take the one
with the smallest Odd-rule number.
These 48 combinatorially inequivalent neighborhoods can be seen in
Tables \ref{Tab1a}, \ref{Tab1b}, \ref{Tab1c}, where they are distinguished by 
having the sequence numbers in \cite{OEIS} for the
$a_n(F)$ and $b_n(F)$ sequences in the final column of the tables.
If the $a_n(F)$ and $b_n(F)$ sequences are the same as those
for some earlier rule, this is
 indicated in the final column instead of the sequence numbers.

The generating functions for the 48 $b_n(F)$ sequences,
together with the corresponding sequence numbers,
are given in the Appendix.  They are shown in such a way that they 
can be easily copied into a computer algebra system (that is, they are given
in a linear rather than two-dimensional format).

In particular, the generating functions for the Fredkin and von Neumann-with-center  CAs
match those derived in \cite{Fredkin}, and so 
provide an alternative proof for Theorems $4$ and $5$ of that paper.

\section{ Further topics }\label{Sec5}
 \subsection{The highest growth rate.}\label{Sec51} 
 It is natural to ask
 which rule produces the greatest number of ON cells.
 We just consider the number that are ON at
 generation $2^n-1$, that is, the subsequence $b_n(F) = a_{2^n-1}(F)$,
 since by Theorem~\ref{Th2} these are local maxima of the $a_n(F)$ sequence, 
 and all other values of $a_n(F)$ are products
 of these values. 
 
 The most fecund rule is somewhat of a surprise: it
 is Odd-rule 365, seen in the top left figure in Fig.\ \ref{Fig1c}.
 This is the unique winner, well ahead of the more 
 obvious candidates such as rules 252, 272, 525, 757, or 777.
 
 For Odd-rule 365 the neighborhood is $F = 1/(xy)+1/x+x/y+1+y+xy$,
 $b_n(F) = 3.4^n-2.3^n, n \ge 0$, with
 generating function $(1-x)/((1-3x)(1-4x))$, recurrence
 $b_{n+1}=7b_n-12b_{n-1}$, and initial values
 $$
 1, \, 6, \, 30, \, 138, \, 606, \, 2586, \, 10830, \, 44778, \, 183486, \, 747066, \, 3027630, \, \ldots \quad (\seqnum{A255463}) 
$$
Other rules do better at the start, but for $n \ge 4$ Odd-rule 365 is the winner, and thus,
for any height-one odd-rule neighborhood $F$, 
\beql{Eq365a}
b_n(F) ~\le  ~ 3.4^n ~-~ 2.3^n \quad \mbox{~for~all~} n \ge 4.
\eeq
Equality holds in \eqn{Eq365a} if and only if $F$ is trivially equivalent to Odd-rule $365$.

After $2^n-1$ generations of any odd-rule height-one CA that starts with a single
ON cell at generation $0$, the ON cells
are contained in the  square of side $2^{n+1}-1$
centered at the origin. Odd-rule 365 turns ON a fraction
\beql{Eq365e}
\frac{ 3.4^n - 2.3^n}{(2^{n+1}-1)^2}
\eeq
of these, which approaches 3/4 as $ n \rightarrow \infty$.

Figure \ref{Fig365} shows generation $15$  of this CA, containing $a_{15}(F) = b_4(F) = 606$
ON cells, and Fig.\ \ref{Fig365b} (to be read from right to left,
top to bottom) shows the evolution of this automaton up to this point. The ON cells
in all these figures are colored black.

From Theorem 3 of \cite{Fredkin}, the $a_n(F)$ sequence, which has initial terms
$$
1, \, 6, \, 6, \, 30, \, 6, \, 36, \, 30, \, 138, \, 6, \, 36, \, 36, \, 180, \, 30, \, 180, \, 138, \, 606, \, 6, \, 36, \, 36, \, 180, \, \ldots \quad (\seqnum{A255462})\, ,
$$
satisfies the recurrence $a_{2t}=a_t$, $ a_{4t+1}=6a_t$,
$a_{4t+3}=7a_{2t+1}-12a_t$ for $t > 0$, with $a_0=1$.

Incidentally, the runner-up is Odd-rule 537, for which
the fraction of ON cells at generations $2^n-1$ approaches 2/3.

\subsection{The lowest growth rate.}\label{Sec52}
Odd-rules 000, 001, 003, 007 have $b_n$ equal to $0$, $1$, $2^n$,
and $(2^{n+2}-(-1)^n)/3$, respectively. But the slowest-growing properly 
\textit{two-dimensional}
rule is Odd-rule $013$, for which
$b_n = 3^n$. 
Figure \ref{Fig013} (drawn at the same scale as Fig.\ \ref{Fig365}) shows generation $15$, containing a mere $a_{15} = b_4 = 81$ ON cells.

\subsection{Explaining combinatorial equivalence.}\label{Sec53}
In some cases it is easy to explain why two different neighborhoods 
have the same $a_n$ (and $b_n$) sequences, i.e., are
combinatorially equivalent.
Let us denote combinatorial equivalence by $\sim$.
All five of the trivially inequivalent two-celled neighborhoods are
combinatorially equivalent---for example, 
rule $003$, $1/y+x/y \sim 1+x \sim 1/y +x$, which is rule $012$.
To see that the four-celled rules $033$ ($1 + x + 1/y + x/y$) and
$505$ ($y/x + xy + 1/(xy) + x/y$) are equivalent,
replace $x$ by $x^2$ in the former, then divide by $x$, replace
$y$ by $y^2$, and finally multiply by $y$.
For other pairs, such as the seven-celled rules $376$ and $557$,
there does not seem to be a simple proof that they are 
combinatorially equivalent, even though we know (by
the theory developed in \cite{ESZ}) that this is true. 

\vspace*{.5in}
 
\begin{figure}[ht]
\begin{minipage}[b]{0.45\linewidth}
\centering
\includegraphics[width=\textwidth]{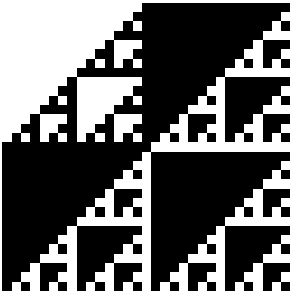}
\caption{Odd-rule 365, the fastest-growing, after 15 generations (see
also Fig\ \ref{Fig365b}).
There are $606$ ON cells.}
\label{Fig365}
\end{minipage}
\hspace{0.4cm}
\begin{minipage}[b]{0.45\linewidth}
\centering
\includegraphics[width=\textwidth]{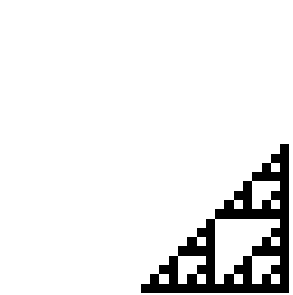}
\caption{Odd-rule 013, the slowest-growing, after 15 generations
(on the same scale).
There are $81$ ON cells.}
\label{Fig013}
\end{minipage}
\end{figure}

\begin{figure}[!ht]
\centerline{\includegraphics[angle=0, width=6.3in]{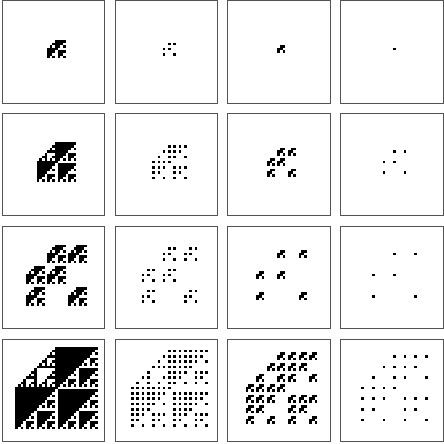}}
\caption{Generations $0$ to $15$ of Odd-rule $365$ (to be read from right to left,
top to bottom).}
\label{Fig365b}
\end{figure}

\begin{table}[!ht]
\caption{ Tables \ref{Tab1a}, \ref{Tab1b}, \ref{Tab1c}
show the 86 trivially inequivalent  neighborhoods
and the 48 combinatorially inequivalent ones.
The asterisk denotes multiplication.
See text for further details. }
\label{Tab1a}
$$
\begin{array}{|c|c|c|c|c|}
\hline
\mbox{Rule} & \mbox{Cells} & \mbox{Neighborhood} & F & a_n(F), ~b_n(F) \\
\hline

000 &	0 &  [0, 0, 0, 0, 0, 0, 0, 0, 0] &   0 &  \seqnum{A000004}, \seqnum{A000004} \\
\hline
001 &	1 &  [0, 0, 0, 0, 0, 0, 0, 0, 1] &   x/y &  \seqnum{A000012}, \seqnum{A000012} \\
\hline
003 &	2 &  [0, 0, 0, 0, 0, 0, 0, 1, 1] &   1/y+x/y &  \seqnum{A001316}, \seqnum{A000079} \\
005 &	2 &  [0, 0, 0, 0, 0, 0, 1, 0, 1] &   1/(x*y)+x/y &  \mbox{=~Odd-rule~} 003 \\
012 &	2 &  [0, 0, 0, 0, 0, 1, 0, 1, 0] &   x+1/y &  \mbox{=~Odd-rule~} 003 \\
014 &	2 &  [0, 0, 0, 0, 0, 1, 1, 0, 0] &   x+1/(x*y) &  \mbox{=~Odd-rule~} 003 \\
104 &	2 &  [0, 0, 1, 0, 0, 0, 1, 0, 0] &   x*y+1/(x*y) &  \mbox{=~Odd-rule~} 003 \\
\hline
007 &	3 &  [0, 0, 0, 0, 0, 0, 1, 1, 1] &   1/(x*y)+1/y+x/y &  \seqnum{A071053}, \seqnum{A001045} \\
013 &	3 &  [0, 0, 0, 0, 0, 1, 0, 1, 1] &   x+1/y+x/y &  \seqnum{A048883}, \seqnum{A000244} \\
015 &	3 &  [0, 0, 0, 0, 0, 1, 1, 0, 1] &   x+1/(x*y)+x/y &  \mbox{=~Odd-rule~} 013 \\
016 &	3 &  [0, 0, 0, 0, 0, 1, 1, 1, 0] &   x+1/(x*y)+1/y &  \mbox{=~Odd-rule~} 013 \\
025 &	3 &  [0, 0, 0, 0, 1, 0, 1, 0, 1] &   1+1/(x*y)+x/y &  \mbox{=~Odd-rule~} 013 \\
105 &	3 &  [0, 0, 1, 0, 0, 0, 1, 0, 1] &   x*y+1/(x*y)+x/y &  \mbox{=~Odd-rule~} 013 \\
106 &	3 &  [0, 0, 1, 0, 0, 0, 1, 1, 0] &   x*y+1/(x*y)+1/y &  \mbox{=~Odd-rule~} 013 \\
124 &	3 &  [0, 0, 1, 0, 1, 0, 1, 0, 0] &   1+x*y+1/(x*y) &  \mbox{=~Odd-rule~} 007 \\
141 &	3 &  [0, 0, 1, 1, 0, 0, 0, 0, 1] &   x*y+1/x+x/y &  \mbox{=~Odd-rule~} 013 \\
142 &	3 &  [0, 0, 1, 1, 0, 0, 0, 1, 0] &   x*y+1/x+1/y &  \mbox{=~Odd-rule~} 013 \\
\hline
017 &	4 &  [0, 0, 0, 0, 0, 1, 1, 1, 1] &   x+1/(x*y)+1/y+x/y &  \seqnum{A253064}, \seqnum{A087206} \\
027 &	4 &  [0, 0, 0, 0, 1, 0, 1, 1, 1] &   1+1/(x*y)+1/y+x/y &  \mbox{=~Odd-rule~} 017 \\
033 &	4 &  [0, 0, 0, 0, 1, 1, 0, 1, 1] &   1+x+1/y+x/y &  \seqnum{A102376}, \seqnum{A000302} \\
035 &	4 &  [0, 0, 0, 0, 1, 1, 1, 0, 1] &   1+x+1/(x*y)+x/y &  \seqnum{A255297}, \seqnum{A027649} \\
036 &	4 &  [0, 0, 0, 0, 1, 1, 1, 1, 0] &   1+x+1/(x*y)+1/y &  \mbox{=~Odd-rule~} 033 \\
055 &	4 &  [0, 0, 0, 1, 0, 1, 1, 0, 1] &   1/x+x+1/(x*y)+x/y &  \mbox{=~Odd-rule~} 033 \\
107 &	4 &  [0, 0, 1, 0, 0, 0, 1, 1, 1] &   x*y+1/(x*y)+1/y+x/y &  \mbox{=~Odd-rule~} 017 \\
116 &   4 &  [0, 0, 1, 0, 0, 1, 1, 1, 0] &   x*y+x+1/(x*y)+1/y &  \mbox{=~Odd-rule~} 035 \\
125 &   4 &  [0, 0, 1, 0, 1, 0, 1, 0, 1] &   1+x*y+1/(x*y)+x/y &  \mbox{=~Odd-rule~} 017 \\
126 &   4 &  [0, 0, 1, 0, 1, 0, 1, 1, 0] &   1+x*y+1/(x*y)+1/y &  \mbox{=~Odd-rule~} 017 \\
143 &   4 &  [0, 0, 1, 1, 0, 0, 0, 1, 1] &   x*y+1/x+1/y+x/y &  \seqnum{A255298}, \seqnum{A255299} \\
145 &   4 &  [0, 0, 1, 1, 0, 0, 1, 0, 1] &   x*y+1/x+1/(x*y)+x/y &  \mbox{=~Odd-rule~} 035 \\
146 &   4 &  [0, 0, 1, 1, 0, 0, 1, 1, 0] &   x*y+1/x+1/(x*y)+1/y &  \seqnum{A255302}, \seqnum{A255303} \\
151 &   4 &  [0, 0, 1, 1, 0, 1, 0, 0, 1] &   x*y+1/x+x+x/y &  \mbox{=~Odd-rule~} 017 \\
152 &   4 &  [0, 0, 1, 1, 0, 1, 0, 1, 0] &   x*y+1/x+x+1/y &  \mbox{=~Odd-rule~} 146 \\
154 &   4 &  [0, 0, 1, 1, 0, 1, 1, 0, 0] &   x*y+1/x+x+1/(x*y) &  \mbox{=~Odd-rule~} 033 \\
161 &   4 &  [0, 0, 1, 1, 1, 0, 0, 0, 1] &   1+x*y+1/x+x/y &  \seqnum{A255300}, \seqnum{A255301} \\
162 &   4 &  [0, 0, 1, 1, 1, 0, 0, 1, 0] &   1+x*y+1/x+1/y &  \mbox{=~Odd-rule~} 033 \\
252 &   4 &  [0, 1, 0, 1, 0, 1, 0, 1, 0] &   y+1/x+x+1/y &  \mbox{=~Odd-rule~} 033 \\
505 &   4 &  [1, 0, 1, 0, 0, 0, 1, 0, 1] &   y/x+x*y+1/(x*y)+x/y &  \mbox{=~Odd-rule~} 033 \\
\hline
\end{array}
$$
\end{table}

\begin{table}[!ht]
\caption{ Tables \ref{Tab1a}, \ref{Tab1b}, \ref{Tab1c}
show the 86 trivially inequivalent  neighborhoods
and the 48 combinatorially inequivalent ones.
See text for further details. }
\label{Tab1b}
$$
\begin{array}{|c|c|c|c|c|}
\hline
\mbox{Rule} & \mbox{Cells} & \mbox{Neighborhood} & F & a_n(F), ~b_n(F) \\
\hline

037 &	5 &  [0, 0, 0, 0, 1, 1, 1, 1, 1] &   1+x+1/(x*y)+1/y+x/y &  \seqnum{A255445}, \seqnum{A001834} \\
057 &	5 &  [0, 0, 0, 1, 0, 1, 1, 1, 1] &   1/x+x+1/(x*y)+1/y+x/y &  \seqnum{A072272}, \seqnum{A007483} \\
117 &	5 &  [0, 0, 1, 0, 0, 1, 1, 1, 1] &   x*y+x+1/(x*y)+1/y+x/y &  \seqnum{A255304}, \seqnum{A255442} \\
127 &	5 &  [0, 0, 1, 0, 1, 0, 1, 1, 1] &   1+x*y+1/(x*y)+1/y+x/y &  \mbox{=~Odd-rule~} 117 \\
136 &	5 &  [0, 0, 1, 0, 1, 1, 1, 1, 0] &   1+x*y+x+1/(x*y)+1/y &  \mbox{=~Odd-rule~} 037 \\
147 &	5 &  [0, 0, 1, 1, 0, 0, 1, 1, 1] &   x*y+1/x+1/(x*y)+1/y+x/y &  \seqnum{A255443}, \seqnum{A255444} \\
153 &	5 &  [0, 0, 1, 1, 0, 1, 0, 1, 1] &   x*y+1/x+x+1/y+x/y &  \seqnum{A255454}, \seqnum{A255455} \\
155 &	5 &  [0, 0, 1, 1, 0, 1, 1, 0, 1] &   x*y+1/x+x+1/(x*y)+x/y &  \mbox{=~Odd-rule~} 037 \\
156 &	5 &  [0, 0, 1, 1, 0, 1, 1, 1, 0] &   x*y+1/x+x+1/(x*y)+1/y &  \seqnum{A255452}, \seqnum{A255453} \\
163 &	5 &  [0, 0, 1, 1, 1, 0, 0, 1, 1] &   1+x*y+1/x+1/y+x/y &  \seqnum{A255456}, \seqnum{A255457} \\
165 &	5 &  [0, 0, 1, 1, 1, 0, 1, 0, 1] &   1+x*y+1/x+1/(x*y)+x/y &  \seqnum{A255446}, \seqnum{A255447} \\
166 &	5 &  [0, 0, 1, 1, 1, 0, 1, 1, 0] &   1+x*y+1/x+1/(x*y)+1/y &  \seqnum{A255450}, \seqnum{A255451} \\
171 &	5 &  [0, 0, 1, 1, 1, 1, 0, 0, 1] &   1+x*y+1/x+x+x/y &  \seqnum{A253065}, \seqnum{A253067} \\
172 &	5 &  [0, 0, 1, 1, 1, 1, 0, 1, 0] &   1+x*y+1/x+x+1/y &  \mbox{=~Odd-rule~} 166 \\
174 &	5 &  [0, 0, 1, 1, 1, 1, 1, 0, 0] &   1+x*y+1/x+x+1/(x*y) &  \mbox{=~Odd-rule~} 057 \\
253 &	5 &  [0, 1, 0, 1, 0, 1, 0, 1, 1] &   y+1/x+x+1/y+x/y &  \mbox{=~Odd-rule~} 156 \\
255 &	5 &  [0, 1, 0, 1, 0, 1, 1, 0, 1] &   y+1/x+x+1/(x*y)+x/y &  \seqnum{A255458}, \seqnum{A255459} \\
272 &	5 &  [0, 1, 0, 1, 1, 1, 0, 1, 0] &   1+y+1/x+x+1/y &  \mbox{=~Odd-rule~} 057 \\
345 &	5 &  [0, 1, 1, 1, 0, 0, 1, 0, 1] &   y+x*y+1/x+1/(x*y)+x/y &  \seqnum{A255448}, \seqnum{A255449} \\
507 &	5 &  [1, 0, 1, 0, 0, 0, 1, 1, 1] &   y/x+x*y+1/(x*y)+1/y+x/y &  \mbox{=~Odd-rule~} 057 \\
525 &	5 &  [1, 0, 1, 0, 1, 0, 1, 0, 1] &   1+y/x+x*y+1/(x*y)+x/y &  \mbox{=~Odd-rule~} 057 \\
\hline
077 &	6 &  [0, 0, 0, 1, 1, 1, 1, 1, 1] &   1+1/x+x+1/(x*y)+1/y+x/y &  \seqnum{A246037}, \seqnum{A246036} \\
137 &	6 &  [0, 0, 1, 0, 1, 1, 1, 1, 1] &   1+x*y+x+1/(x*y)+1/y+x/y &  \seqnum{A255464}, \seqnum{A255465} \\
157 &	6 &  [0, 0, 1, 1, 0, 1, 1, 1, 1] &   x*y+1/x+x+1/(x*y)+1/y+x/y &  \seqnum{A255468}, \seqnum{A255469} \\
167 &	6 &  [0, 0, 1, 1, 1, 0, 1, 1, 1] &   1+x*y+1/x+1/(x*y)+1/y+x/y &  \seqnum{A255466}, \seqnum{A255467} \\
173 &	6 &  [0, 0, 1, 1, 1, 1, 0, 1, 1] &   1+x*y+1/x+x+1/y+x/y &  \seqnum{A255475}, \seqnum{A255476} \\
175 &	6 &  [0, 0, 1, 1, 1, 1, 1, 0, 1] &   1+x*y+1/x+x+1/(x*y)+x/y &  \seqnum{A253069}, \seqnum{A253070} \\
176 &	6 &  [0, 0, 1, 1, 1, 1, 1, 1, 0] &   1+x*y+1/x+x+1/(x*y)+1/y &  \seqnum{A255470}, \seqnum{A255471} \\
257 &	6 &  [0, 1, 0, 1, 0, 1, 1, 1, 1] &   y+1/x+x+1/(x*y)+1/y+x/y &  \seqnum{A255473}, \seqnum{A255474} \\
273 &	6 &  [0, 1, 0, 1, 1, 1, 0, 1, 1] &   1+y+1/x+x+1/y+x/y &  \mbox{=~Odd-rule~} 176 \\
275 &	6 &  [0, 1, 0, 1, 1, 1, 1, 0, 1] &   1+y+1/x+x+1/(x*y)+x/y &  \seqnum{A253066}, \seqnum{A253068} \\
347 &	6 &  [0, 1, 1, 1, 0, 0, 1, 1, 1] &   y+x*y+1/x+1/(x*y)+1/y+x/y &  \seqnum{A253100}, \seqnum{A253101} \\
356 &	6 &  [0, 1, 1, 1, 0, 1, 1, 1, 0] &   y+x*y+1/x+x+1/(x*y)+1/y &  \seqnum{A247640}, \seqnum{A164908} \\
365 &	6 &  [0, 1, 1, 1, 1, 0, 1, 0, 1] &   1+y+x*y+1/x+1/(x*y)+x/y &  \seqnum{A255462}, \seqnum{A255463} \\
517 &	6 &  [1, 0, 1, 0, 0, 1, 1, 1, 1] &   y/x+x*y+x+1/(x*y)+1/y+x/y &  \seqnum{A255460}, \seqnum{A255461} \\
527 &	6 &  [1, 0, 1, 0, 1, 0, 1, 1, 1] &   1+y/x+x*y+1/(x*y)+1/y+x/y &  \seqnum{A255295}, \seqnum{A255296} \\
555 &	6 &  [1, 0, 1, 1, 0, 1, 1, 0, 1] &   y/x+x*y+1/x+x+1/(x*y)+x/y &  \mbox{=~Odd-rule~} 077 \\
\hline
\end{array}
$$
\end{table}

\begin{table}[!ht]
\caption{ Tables \ref{Tab1a}, \ref{Tab1b}, \ref{Tab1c}
show the 86 trivially inequivalent  neighborhoods
and the 48 combinatorially inequivalent ones.
See text for further details. }
\label{Tab1c}
$$
\begin{array}{|c|c|c|c|c|}
\hline
\mbox{Rule} & \mbox{Cells} & \mbox{Neighborhood} & F & a_n(F),~ b_n(F) \\
\hline

177 &   7 &  [0, 0, 1, 1, 1, 1, 1, 1, 1] &   1+x*y+1/x+x+1/(x*y) & \\
&&& +1/y+x/y &  \seqnum{A255277}, \seqnum{A255278} \\
277 &   7 &  [0, 1, 0, 1, 1, 1, 1, 1, 1] &   1+y+1/x+x+1/(x*y) & \\
&&& +1/y+x/y &  \seqnum{A255279}, \seqnum{A255280} \\
357 &   7 &  [0, 1, 1, 1, 0, 1, 1, 1, 1] &   y+x*y+1/x+x+1/(x*y) & \\
&&& +1/y+x/y &  \seqnum{A253071}, \seqnum{A253072} \\
367 &   7 &  [0, 1, 1, 1, 1, 0, 1, 1, 1] &   1+y+x*y+1/x+1/(x*y) & \\
&&& +1/y+x/y &  \seqnum{A255281}, \seqnum{A255282} \\
376 &   7 &  [0, 1, 1, 1, 1, 1, 1, 1, 0] &   1+y+x*y+1/x+x & \\
&&& +1/(x*y)+1/y &  \seqnum{A247666}, \seqnum{A102900} \\
537 &   7 &  [1, 0, 1, 0, 1, 1, 1, 1, 1] &   1+y/x+x*y+x+1/(x*y) & \\
&&& +1/y+x/y &  \seqnum{A255283}, \seqnum{A255284} \\
557 &   7 &  [1, 0, 1, 1, 0, 1, 1, 1, 1] &   y/x+x*y+1/x+x & \\
&&& +1/(x*y)+1/y+x/y &  \mbox{=~Odd-rule~} 376 \\
575 &   7 &  [1, 0, 1, 1, 1, 1, 1, 0, 1] &   1+y/x+x*y+1/x+x & \\
&&& +1/(x*y)+x/y &  \seqnum{A246039}, \seqnum{A246038} \\
\hline
377 &   8 &  [0, 1, 1, 1, 1, 1, 1, 1, 1] &   1+y+x*y+1/x+x & \\
&&& +1/(x*y)+1/y+x/y &  \seqnum{A255275}, \seqnum{A255276} \\
577 &   8 &  [1, 0, 1, 1, 1, 1, 1, 1, 1] &   1+y/x+x*y+1/x+x & \\
&&& +1/(x*y)+1/y+x/y &  \seqnum{A253104}, \seqnum{A253105} \\
757 &   8 &  [1, 1, 1, 1, 0, 1, 1, 1, 1] &   y/x+y+x*y+1/x+x & \\
&&& +1/(x*y)+1/y+x/y &  \seqnum{A160239}, \seqnum{A246030} \\
\hline
777 &   9 &  [1, 1, 1, 1, 1, 1, 1, 1, 1] &   (1/x+1+x) & \\
&&& *(1/y+1+y) &  \seqnum{A246035}, \seqnum{A139818} \\
\hline
\end{array}
$$
\end{table}

\clearpage

\section*{Appendix}\label{SecA}

For each of the 48 combinatorially inequivalent height-one
neighborhoods (see Theorem \ref{Th4} and Tables \ref{Tab1a},
\ref{Tab1b}, \ref{Tab1c})
this Appendix gives the Odd-rule number, the number
in \cite{OEIS} of the $b_n(F)$ sequence,
and a generating function for that sequence.
(In two or three cases, for example Odd-rule 007, the sequence in \cite{OEIS} 
has an extra initial term compared with the $b_n(F)$ sequence, so the
generating function given here is not exactly the same as the one in \cite{OEIS}.) 

\begin{verbatim}
Zero cells:
000: (A000004) 1

One cell:
001: (A000012) 1/(1-x)

Two cells:
003: (A000079) 1/(1-2*x)

Three cells:
007: (A001045) (1+2*x)/((1+x)*(1-2*x))
013: (A000244) 1/(1-3*x)

Four cells:
017: (A087206) (1+2*x)/(1-2*x-4*x^2)
033: (A000302) 1/(1-4*x)
035: (A027649) (1-x)/((1-2*x)*(1-3*x))
143: (A255299) (1-x)*(1-x+x^2-x^3-4*x^4+2*x^5-2*x^6)
          /(1-6*x+10*x^2-4*x^3-3*x^4+12*x^5-20*x^6+10*x^7-4*x^8)
146: (A255303) (1-x+2*x^2+2*x^3)/((1-3*x-2*x^2)*(1-2*x+2*x^2))
161: (A255301) (1-x)*(1+x+2*x^2)/(1-4*x+x^2+2*x^3+4*x^4)

Five cells:
037: (A001834) (1+x)/(1-4*x+x^2)
057: (A007483) (1+2*x)/(1-3*x-2*x^2)
117: (A255442) (1+3*x)*(1-x)/((1-3*x)*(1-3*x^2))
147: (A255444) (1-x)*(1+2*x+7*x^4+4*x^5+2*x^6)
          /(1-4*x-x^2+8*x^3+7*x^4-26*x^5+11*x^6+14*x^7+2*x^8-4*x^9)
153: (A255455) (1-x-5*x^2+9*x^3-12*x^4+14*x^5-4*x^6+8*x^7)
          /(1-6*x+6*x^2+20*x^3-51*x^4+56*x^5-46*x^6+20*x^7-8*x^8)
156: (A255453) (1-x+2*x^2-4*x^3)/((1-x)*(1-5*x+6*x^2-4*x^3))
163: (A255457) (1-x)*(1+x-x^2+x^3)/(1-5*x+24*x^3-15*x^4-17*x^5)
165: (A255447) (1-x)*(1+x)*(1+x-x^2)/((1-x-x^2)*(1-3*x-5*x^2+11*x^3))
166: (A255451) (1+x)/(1-4*x-x^2+4*x^3+8*x^4)
171: (A253067) (1+2*x)*(1+2*x+3*x^2+4*x^3)/(1-x-5*x^2-13*x^3-6*x^4-8*x^5)
255: (A255459) (1-x+6*x^2)/((1-x)*(1-2*x)*(1-3*x))
345: (A255449) (1-x)*(1-x^2-2*x^3-6*x^4)/(1-6*x+10*x^2-8*x^3+15*x^4-10*x^5-10*x^6)

Six cells:
077: (A246036) (1+4*x)/((1+2*x)*(1-4*x))
137: (A255465) (1+3*x)/((1+x)*(1-4*x))
157: (A255469) (1+x-9*x^2+15*x^3+2*x^4-34*x^5+20*x^6-16*x^7-8*x^8)
          /(1-5*x+x^2+25*x^3-44*x^4+2*x^5+56*x^6-40*x^7+24*x^8+16*x^9)
167: (A255467) (1+2*x)*(1-x)/((1-4*x)*(1-2*x)*(1+x))
173: (A255476) (1+2*x)*(1+x-2*x^2-x^3-5*x^4-7*x^5+2*x^6-7*x^7+6*x^8)
          /(1-3*x-6*x^2+9*x^3+9*x^4+7*x^5-2*x^6-15*x^7+4*x^8+8*x^9+8*x^10)
175: (A253070) (1+2*x)*(1+x-x^2+x^3+2*x^5)
          /(1-3*x-3*x^2+x^3+6*x^4-10*x^5+8*x^6-8*x^7)
176: (A255471) (1+3*x)/((1-x)*(1+2*x)*(1-4*x))
257: (A255474) (1-8*x^2-16*x^3)/((1-4*x)*(1-2*x-4*x^2))
275: (A253068) (1+3*x+4*x^2)/((1-x)*(1+2*x)*(1-4*x))
347: (A253101) (1-3*x^2+4*x^3)/((1-2*x)*(1-4*x+x^2))
356: (A164908) (1+2*x)/(1-4*x)
365: (A255463) (1-x)/((1-3*x)*(1-4*x))
517: (A255461) (1-x)*(1+2*x-3*x^2-6*x^3-2*x^4+4*x^5)
          /((1+x)*(1-2*x)*(1-4*x+x^2+2*x^4-4*x^5))
555: (A255296) (1+2*x)*(1-x)/((1-2*x)*(1-3*x-2*x^2))

Seven cells:
177: (A255278) (1+4*x-3*x^3+6*x^4-6*x^5-12*x^6-12*x^7)
          /((1+x)*(1-4*x-2*x^2+9*x^3+2*x^4+2*x^5-8*x^6+12*x^7))
277: (A255280) (1+2*x-7*x^2+12*x^3-16*x^5-16*x^6)
          /(1-5*x+3*x^2+9*x^3-16*x^4+16*x^6+16*x^7)
357: (A253072) (1+x-16*x^2+28*x^3-8*x^4)/(1-6*x+5*x^2+24*x^3-44*x^4+8*x^5)
367: (A255282) (1+2*x-4*x^2-7*x^3+5*x^4-2*x^5+9*x^7-2*x^8+6*x^9)
          /(1-5*x+21*x^3-18*x^4-3*x^5+24*x^6-31*x^7+11*x^8-22*x^9-10*x^10)
376: (A102900) (1+4*x)/((1+x)*(1-4*x))
537: (A255284) (1+4*x)*(1-x)/((1-4*x)*(1-7*x^2))
575: (A246038) (1+2*x)*(1+2*x+4*x^2)/(1-3*x-8*x^3-8*x^4)

Eight cells:
377: (A255276) (1-20*x^2+56*x^3-49*x^4-36*x^5+128*x^6-128*x^7)
          /(1-8*x+16*x^2+24*x^3-145*x^4+236*x^5-164*x^6+24*x^7+32*x^8)
577: (A253105) (1+4*x+6*x^3+x^4+8*x^5-4*x^6)
          /(1-4*x-6*x^3+29*x^4-12*x^5-8*x^6+8*x^7)
757: (A246030) (1+6*x)/((1+2*x)*(1-4*x))

Nine cells:
777: (A139818) (1+6*x-8*x^2)/((1-x)*(1+2*x)*(1-4*x))

\end{verbatim}

\bigskip
\hrule
\bigskip

\noindent 2010 {\it Mathematics Subject Classification}:
Primary 11B85, 37B15.

\noindent \emph{Keywords: } cellular automaton, 
 Moore neighborhood,  von Neumann neighborhood,
Odd-rule cellular automaton,  run length transform, 
Fredkin Replicator, Rule 110, Rule 150

\bigskip
\hrule
\bigskip

\noindent (Concerned with sequences
\seqnum{A000004},
\seqnum{A000012},
\seqnum{A000079},
\seqnum{A000244},
\seqnum{A000302},
\seqnum{A001045},
\seqnum{A001316},
\seqnum{A001834},
\seqnum{A007483},
\seqnum{A027649},
\seqnum{A048883},
\seqnum{A071053},
\seqnum{A072272},
\seqnum{A087206},
\seqnum{A102376},
\seqnum{A102900},
\seqnum{A139818},
\seqnum{A160239},
\seqnum{A164908},
\seqnum{A245562.},
\seqnum{A246030},
\seqnum{A246035},
\seqnum{A246036},
\seqnum{A246037},
\seqnum{A246038},
\seqnum{A246039},
\seqnum{A247640},
\seqnum{A247666},
\seqnum{A253064},
\seqnum{A253065},
\seqnum{A253066},
\seqnum{A253067},
\seqnum{A253068},
\seqnum{A253069},
\seqnum{A253070},
\seqnum{A253071},
\seqnum{A253072},
\seqnum{A253100},
\seqnum{A253101},
\seqnum{A253104},
\seqnum{A253105},
\seqnum{A255275},
\seqnum{A255276},
\seqnum{A255277},
\seqnum{A255278},
\seqnum{A255279},
\seqnum{A255280},
\seqnum{A255281},
\seqnum{A255282},
\seqnum{A255283},
\seqnum{A255284},
\seqnum{A255295},
\seqnum{A255296},
\seqnum{A255297},
\seqnum{A255298},
\seqnum{A255299},
\seqnum{A255300},
\seqnum{A255301},
\seqnum{A255302},
\seqnum{A255303},
\seqnum{A255304},
\seqnum{A255442},
\seqnum{A255443},
\seqnum{A255444},
\seqnum{A255445},
\seqnum{A255446},
\seqnum{A255447},
\seqnum{A255448},
\seqnum{A255449},
\seqnum{A255450},
\seqnum{A255451},
\seqnum{A255452},
\seqnum{A255453},
\seqnum{A255454},
\seqnum{A255455},
\seqnum{A255456},
\seqnum{A255457},
\seqnum{A255458},
\seqnum{A255459},
\seqnum{A255460},
\seqnum{A255461},
\seqnum{A255462},
\seqnum{A255463},
\seqnum{A255464},
\seqnum{A255465},
\seqnum{A255466},
\seqnum{A255467},
\seqnum{A255468},
\seqnum{A255469},
\seqnum{A255470},
\seqnum{A255471},
\seqnum{A255473},
\seqnum{A255474},
\seqnum{A255475},
\seqnum{A255476}.)

\bigskip
\hrule
\bigskip

\end{document}